\documentclass{amsart}

\usepackage{xcolor,comment,amssymb,enumerate,comment,ulem}
\usepackage[verbose]{hyperref}
\usepackage{xcolor,comment,amssymb,enumerate,comment,cleveref,ulem}
\hypersetup{colorlinks=false,allbordercolors=blue,pdfborderstyle={/S/U/W 1}}
\usepackage[sort,compress]{cite}
\usepackage{url}

\renewcommand{\email}[1]{{\it #1}\par}

\DeclareMathOperator{\Aut}{Aut}
\DeclareMathOperator{\Gr}{Gr}
\DeclareMathOperator{\QGr}{QGr}
 
\DeclareMathOperator{\gldim}{gldim}

\DeclareMathOperator{\id}{id}

\renewcommand{\k}{{\Bbbk}}
\renewcommand{\P}{\mathbb{P}}

\newcommand{\kk}{\Bbbk}

\newcommand{\V}{\mathcal{V}}

\newtheorem{thm}{Theorem}[section]
\newtheorem{lem}[thm]{Lemma}
\newtheorem{cor}[thm]{Corollary}
\newtheorem{prop}[thm]{Proposition}

\newtheorem{notn}[thm]{Notation}
\newtheorem{ex}[thm]{Example}
\newtheorem{rmk}[thm]{Remark}

\newtheorem{Pf}{Proof$\!\!$}

\DeclareMathSymbol{\twoheadrightarrow}  {\mathrel}{AMSa}{"10}
\newcounter{letter}
\renewcommand{\theletter}{\rom{(}\alph{letter}\rom{)}}

\newcounter{rnum}
\renewcommand{\thernum}{\rom{(}\roman{rnum}\rom{)}}

\usepackage{xcolor}

\begin{document}

%\markboth{Chandler, Tran, Veerapen, Wang}{}

%\catchline{}{}{}{}{}
%
%%%%%%%%%%%%%%%%%%%%%%%%%%%%%%%%%%%%%%%%%%%%%%%%%%%%%%%%%%%%%%%%%%%%

\title{Embedding a Quantum Rank Two Quadric in a Quantum $\mathbb{P}^3$}

\author{\footnotesize Richard G.\ Chandler}
\address{(Chandler) University of North Texas at Dallas,
7300 University Hills Blvd,
Dallas, TX 75241.}
\email{richard.chandler@untdallas.edu}

\author{\footnotesize Hung V. Tran}
\address{(Tran) Department of Mathematics,
University of Texas at Arlington,
Arlington, TX 76019}
\email{hvtran01@gmail.com}

\author{\footnotesize Padmini Veerapen}
\address{(Veerapen) Department of Computer and Mathematical Sciences, Indiana University Northwest,\\ Gary, IN 46408, U.S.A.}
\email{pveerap@iu.edu}

\author{\footnotesize Xingting Wang}
\address{(Wang) Department of Mathematics, Louisiana State University, Baton Rouge, Louisiana 70803, U.S.A.}
\email{xingtingwang@math.lsu.edu}

\maketitle

\begin{comment}

%\begin{history}
\received{(Day Month Year)}
\revised{(Day Month Year)}
\accepted{(Day Month Year)}
\published{(Day Month Year)}
\comby{(xxxxxxxxx)}
%\end{history}

%\begin{abstract}
In this paper, we classify connected graded quadratic Artin-Schelter regular (AS-regular, henceforth) algebras of global dimension four that have a Hilbert series the same as that of the polynomial ring on four generators and that map onto a twisted homogeneous coordinate ring of a rank-two quadric. A twisted homogeneous coordinate ring is a construction that was defined by Artin, Tate, and Van den Bergh in the context of the classification of AS-regular algebras of global dimension three. In {prior work, Shelton, Vancliff, and Van Rompay} classified AS-regular algebras of global dimension four that map onto the twisted homogeneous coordinate ring of a rank-three and a rank-four quadric, respectively. We expand on their work to include the case of coordinate rings of a rank-two quadric.
%\end{abstract}

\keywords{Artin-Schelter regular algebra; point scheme; rank-two quadric}

\ccode{2020 Mathematics Subject Classification: 14A22, 16S37, 16S38}
%\subjclass[2020]{16S37, 16S38}%
%14A22 Noncommutative algebraic geometry
%16S37 Quadratic and Koszul algebras
%16S38 Rings arising from non-commutative algebraic geometry
\end{comment}

\section{Introduction}
It is widely accepted that Artin-Schelter regular (AS-regular, henceforth) algebras can be viewed as noncommutative analogues of polynomial rings.  These algebras were defined by Artin and Schelter in 1987 \cite{AS}; Artin, Tate and Van den Bergh \cite{ATV1, ATV2} classified those of global dimension three using techniques that became known as noncommutative projective algebraic geometry. In this paper, our overarching goal is to provide examples that would contribute to the classification of AS-regular algebras of dimension four by utilizing these geometric techniques.

In \cite{AVdB1990}, Artin and Van den Bergh construct a ``noncommutative" version of the commutative homogeneous coordinate ring, known as a twisted homogeneous coordinate ring. As defined in \cite[Definition p. 249]{AVdB1990}, a twisted homogeneous coordinate ring, $S$, is given by three pieces of data, a projective scheme $Q$ defined over a field $\Bbbk$, an automorphism $\sigma$ of $Q$, and an invertible sheaf on $Q$, that is, we may write $S$ as $S(Q, \sigma, \mathcal{L})$. In \cite[Definition 4.3.1]{MSRIBook}, if the invertible sheaf $\mathcal{L}$ is $\sigma$-ample, then such rings are noetherian and finitely generated by \cite[Theorem 1.4]{AVdB1990}.

Shelton, Vancliff, and Van Rompay further examine the twisted homogeneous ring construction in \cite{SV99, VVr} by classifying the quadratic noetherian AS-regular algebras of dimension four that map onto the twisted homogeneous coordinate ring of a rank-three and a rank-four quadric in $\mathbb{P}^3$.  The algebras are classified by considering whether their point schemes $P$ either strictly contains a rank-two quadric, that is, $Q \subset P,$ or are exactly a rank-two quadric, that is, $P = Q$.  In the case of the quadric being strictly contained in the point scheme, it was shown that the algebra is completely determined by the geometric data and, thus, the defining relations can be directly recovered. In this paper, we extend their work to twisted homogeneous coordinate rings of a rank-2 quadric, $Q \subset \mathbb{P}^3$. We also note that algebras of this type have appeared recently in the work of Goetz, Kirkman, Moore, and Vashaw \cite{GKMV}.

Below, $R$ denotes an AS-regular algebra of global dimension four that maps onto a twisted homogeneous coordinate ring $S$ of a rank-two quadric $Q \subset \P^3$ (precise definitions are given in Notation~\ref{notn1}). 

The main results of this paper are as follows. In Section~\ref{sec1}, we establish preliminary results. In particular, Proposition \ref{OmegaNormal} shows that if $A$ is a quadratic domain with $\Omega \in A_2$ such that $A/{\langle \Omega \rangle} = S$, then $\Omega$ is normal; the proof techniques introduced there are used throughout the paper. In Section~\ref{sec2}, we classify the algebras $R$ whose point scheme strictly contains a rank-two quadric (\Cref{PLAlg} and \Cref{PLTwist}). We first find such algebras whose automorphism on the quadric is the identity and then twist these algebras by a non-identity automorphism in $\Aut(\P^3)$ via a Zhang twist. In Section~\ref{sec3}, we classify the algebras $R$ whose point scheme is exactly a rank-two quadric (\Cref{prop1}), and show that such algebras are central extensions of AS-regular algebras of dimension three (Remark~\ref{centralext}).

%%%%%%%%%%%%%%%%%%%%%%%%%%%%%%%%%%%%%%%%%%%%%%%%%%%%%%%%%%%%%%%%%%%%%%
% Preliminaries %Section 2 using wsjaa.cls
%%%%%%%%%%%%%%%%%%%%%%%%%%%%%%%%%%%%%%%%%%%%%%%%%%%%%%%%%%%%%%%%%%%%%%

\section{Preliminaries}\label{sec1}
In this section, we introduce notation and establish preliminary results needed throughout the paper. We set up the framework for studying AS-regular algebras that map onto twisted homogeneous coordinate rings of rank-two quadrics and prove, in \Cref{OmegaNormal}, a normality result for a certain quadratic element $\Omega$.

\begin{notn}\label{notn1}
Throughout, we write $\k^\times$ for the group of nonzero elements of the field $\k$. We let $Q \subset \P^3$ be a rank-two quadric and fix an automorphism $\tau$ of $Q$. By $S(Q)$, we mean the (commutative) homogeneous coordinate ring of $Q$. By $S$, we mean the twisted homogeneous coordinate ring corresponding to a rank-2 quadric $Q$, an automorphism $\tau$ of $Q$ and a $\tau$--ample invertible sheaf $\mathcal{L}$, that is, $S = S(Q, \tau, \mathcal{L})$. The set $\mathcal{R}(Q, \tau)$ consists of all graded algebras $R$ that are quadratic connected Artin-Schelter regular of global dimension four with Hilbert series $(1-t)^{-4}$ and such that there is a graded degree zero homomorphism mapping $R$ onto $S$. We denote $R \in \mathcal{R}(Q, \tau)$ by $T(V)/I$ where $V$ is a four-dimensional vector space, $T(V)$ is the tensor algebra on $V$ and $I$ is a finitely generated graded ideal.
\end{notn}

In \cite[Lemma 1.3]{SV99}, if $A$ is a $\k$-algebra that is Artin-Schelter regular of global dimension four with four linearly independent generators, then $A$ is Koszul and has Hilbert series $H_A(t) = (1-t)^{-4}$. In particular, $A$ is quadratic with six defining relations and has Gelfand-Kirillov dimension four. For any such $\k$-algebra $A$ (including the algebras $R \in \mathcal{R}(Q,\tau)$ introduced above), we write $A = T(V)/\langle W \rangle$ where $W$ is a six-dimensional vector space spanned by the defining relations of $A$. 

\begin{rmk}\label{omega}
\begin{enumerate}
\item Since $Q \subset \P^3$ is a rank-two quadric, it follows that $Q$ is the union of two planes, say $\Pi_1$ and $\Pi_2$. An automorphism $\tau$ of $Q$ either sends each plane to itself or swaps the two planes. In either case, $\tau$ restricts to a linear map on each plane, as all automorphisms of $\P^2$ are linear, so $\tau$ is a linear automorphism of $Q$. In particular, $\tau$ can be lifted to an automorphism of $\P^3$, although such a lift is not unique. We fix a choice of a lift of $\tau$ to $\P^3$ and, by a common abuse of notation, we also denote this lift by $\tau$. When $\tau$ is applied to elements of $R_1$ or $S_1$, it is via this chosen lift. We note that the non-uniqueness of a lift does not arise if $\mathrm{rank}(Q) \in \{3,4\}$ (as in \cite{SV99,VVr}).
\item Using an argument similar to \cite[Lemma 1.7]{SV99}, we note that the Hilbert series of the homogeneous coordinate ring $S(Q')$ of any quadric $Q'  \subset  \P^3$ is $(1+t)(1-t)^{-3}$. This follows from the fact that we may view $S(Q')$ as a quotient of $S(\P^3)$ by the ideal determined by $Q'$, that is, by the ideal generated by the quadratic form defining $Q'$. As a result, the Hilbert series of $S(Q')$ is $(1-t)^{-4} - t^2(1-t)^{-4} = (1+t)(1-t)^{-3}$. In general, we use $S($--$)$ to mean the (commutative) homogeneous coordinate ring of that scheme.
\item By hypothesis, $R$ has Hilbert series $(1-t)^{-4}$ and maps onto $S$ with Hilbert series $(1+t)(1-t)^{-3}$. Similar to \cite[before Definition 1.3]{VVr}, this implies that there exists $\Omega \in R_2$ such that $R/{\langle \Omega \rangle} \cong S$. We note that, since $Q$ is a rank-two quadric, the quadratic form defining $Q$ factors as a product of two linearly independent linear forms; hence, $S$ is \emph{not} a domain (unlike the rank-three and rank-four cases studied in \cite{SV99, VVr}). Nevertheless, there are no nonzero elements of $S_2$ that vanish on the graph of $\tau$, so any nonzero element of $R_2$ that vanishes on the graph of $\tau$ must be a scalar multiple of $\Omega$. Thus, for any $u, v \in R_1$, $u^{\tau}v - v^{\tau}u \in R_2$ and is a scalar multiple of $\Omega$.
\end{enumerate}
\end{rmk}

\begin{ex}\label{omegaNotNormal}
Let $A$ be the $\k$-algebra on generators $x_1,x_2,x_3,x_4$ with the following defining relations $$\begin{matrix} x_1x_2-x_2x_1, & x_1x_3-x_3x_1, & x_1x_4-x_4x_1, & x_2x_4-x_4x_2, & x_3x_4-x_4x_3, & x_1x_2. \end{matrix}$$
This algebra is not noetherian since the following chain of right ideals does not stabilize: $$x_2A \subset x_2A + x_3x_2A \subset x_2A+x_3x_2A+x_3^2x_2A\subset \cdots \subset \sum_{i=0}^n x_3^i x_2 A \subset \cdots$$
However, $A/\langle \Omega \rangle$, where $\Omega=x_2x_3-x_3x_2 \in A$, is the coordinate ring of the rank-two quadric $\V(x_1x_2)$, which is noetherian.  Therefore, \cite[Lemma 8.2]{ATV1} implies that $\Omega$ is not normal in $A$.
\end{ex}

\begin{rmk}\label{zhangtwist}
\begin{enumerate}
\item By Example \ref{omegaNotNormal}, we see that, in general, for a quadratic element $\Omega \in A_2$, $A/{\langle \Omega \rangle} \cong S$ does not ensure that $\Omega$ is normal. This is in agreement with \cite[Example 1.6]{SV99} when the quadric is rank-three but differs from \cite[Lemma 1.5]{VVr} when the quadric is rank-four.
\item 
By \cite[Theorems 3.1 \& 3.5]{Zhang96}, when a connected graded algebra $A$ is twisted by an automorphism or by a twisting system of $A$ denoted by $\phi$, then the category of graded modules of $A$, $\Gr A$, and that of graded modules over $A^{\phi}$, $\Gr A^{\phi}$, are equivalent; that is, the category of graded modules of $A$ are invariant under twists by automorphisms and twisting systems. Additionally, by \cite[Theorem 1.4]{Zhang96}, since $\QGr A$ is equivalent to $\QGr A^{\phi}$, it follows that the point schemes of $A$ and $A^{\phi}$ are equivalent, where $\QGr A$, is the quotient category $\Gr A/\mbox{Tors } A$ and Tors $A$ is the full subcategory of $\Gr A$ consisting of direct limits of finite dimensional modules. The global dimension and Hilbert series of $A$ and $A^{\phi}$ also remain invariant. We refer to the algebras $A$ and $A^{\phi}$ as \textit{Zhang twists} of each other as suggested by other authors (see e.g., \cite{Si}). Recall that the point scheme of $A$ is the projective variety which parametrizes the point modules over $A$. The interested reader may refer to the comprehensive survey by Rogalski \cite{rogalski2024} for additional background.
\end{enumerate}
\end{rmk}

While \Cref{OmegaNormal} is not used explicitly in this paper, the results established in the proof, for example, the fact that $\Omega \in \k^\times X_1 X_2$ and the method used to construct the relations are used heavily throughout the paper.

\begin{prop} \label{OmegaNormal} If $A$ is a quadratic domain  with $\Omega \in A_2$ such that $A/\langle \Omega \rangle = S$, then $\Omega$ is normal in $A$.
\end{prop}
\begin{proof}
In this proof, we use techniques similar to \cite[Lemma 1.7]{SV99} with significant modifications. By Remark \ref{omega}, we may assume that $\tau = \id$, that is, $S=S(Q)$,  and $uv-vu\in \k \Omega$ for all $u,v \in A_1$.
Let $Q=\V(X_1X_2)$, where $X_1, X_2 \in A_1$ are linearly independent, and $X_3 \in A_1 \setminus (\k X_1 \oplus \k X_2)$. It follows that $X_iX_j-X_jX_i=\alpha_{ij}\Omega$ for $\alpha_{ij} \in \k$, $i,j \in \{1,2,3\}$.  Define $y=\alpha_{23}X_1 - \alpha_{13}X_2+\alpha_{12}X_3$.  If all $\alpha_{ij}=0$, then $A$ is commutative on $\k X_1 + \k X_2 + \k X_3$ and we set $y=0$; this case is handled below. Otherwise, it can be verified that $y \neq 0$ and $y$ commutes in $A$ with all elements of $\k X_1+\k X_2+\k X_3$. Note also that $X_1X_2\in \k^\times \Omega$, as it is an element of $A_2$ that vanishes on the graph of the identity restricted to $Q$ and is nonzero since $A$ is a domain; thus, $\Omega\in\k^\times X_1X_2$. Observe that the plane $\V(y)\subset \P^3$ meets $Q$ either in two lines, counted with multiplicity, or in a plane.

If $\V(y)$ meets $Q$ in two lines, counted with multiplicity, then we may present $A$ using generators $x_1=X_1$, $x_2=X_2$, and $x_3=y$ and so, $x_1x_3=x_3x_1$ and $x_2x_3=x_3x_2$. Moreover, if $\alpha_{12} = 0$, that is, if $x_1x_2=x_2x_1$, then $\Omega$ commutes with $x_1$, $x_2$ and $x_3$.  Additionally, if $x_4 \in A_1 \setminus (\k x_1 + \k x_2 + \k x_3)$, then
\begin{align*} x_4 \Omega
&= x_4(\beta x_1 x_2), \qquad \mbox{ where } \Omega = \beta x_1x_2 \\
&= \beta(x_1x_4 + \gamma \Omega) x_2 \\
%&= \beta (x_1x_4x_2 + \gamma \Omega x_2) \\
&= \beta (x_1(x_2 x_4 + \delta \Omega)+\gamma \Omega x_2) \\
%&= \beta(x_1x_2x_4 + \delta x_1 \Omega + \gamma \Omega x_2)\\
%&= \Omega x_4 + \beta \delta \Omega x_1 + \beta \gamma \Omega x_2\\
&= \Omega(x_4+\beta \delta x_1 + \beta \gamma x_2),
\end{align*}
where $\gamma,\delta \in \k$, that is, $\Omega$ is normal in $A$. Otherwise, $\alpha_{12} \ne 0$, that is, $x_1x_2\neq x_2x_1$, and we show there is $x_4 \in A_1 \setminus (\k x_1 + \k x_2 + \k x_3$) such that $x_1x_4 = x_4 x_1$ and $x_2 x_4 = x_4 x_2$.  Choose $w \in A_1 \setminus (\k x_1 + \k x_2 + \k x_3$).  If $w$ commutes with $x_1$, define $z=w$; otherwise $x_1w-wx_1=\beta \Omega$, for $\beta \in \k^\times$.  In this case, define $z=\alpha_{12}w-\beta x_2$; in either case, $z$ commutes with $x_1$.  If, in addition, $z$ commutes with $x_2$, define $x_4=z$.  Otherwise, $x_2z-zx_2=\gamma \Omega$, for some $\gamma \in \k^\times$; define $x_4=\alpha_{12}z + \gamma x_1$. This $x_4$ commutes with $x_2$ (and with $x_1)$, so we may always find an $x_4$ as described above.  It follows that $\Omega$ is normal in $A$.

If $y = 0$ or if $\V(y)$ meets $Q$ in a plane, then we may assume without loss of generality that $y = \alpha_{23} X_1$ for some $\alpha_{23} \in \k^\times$ (if $y=0$, all $\alpha_{ij}=0$, and a similar argument applies by relabeling). That is, $\alpha_{13} = 0 = \alpha_{12}$ and the defining relations of $A$ include $X_1X_2=X_2X_1$, $X_1X_3=X_3X_1$, and $X_2X_3-X_3X_2=\alpha_{23}\Omega$.  Choose $w \in A_1 \setminus (\k X_1 + \k X_2 + \k X_3$).  If $w$ commutes with $X_3$, define $z=w$; otherwise $X_3w-wX_3=\beta \Omega$, for $\beta \in \k^\times$.  In this case, define $z=\alpha_{23}w+\beta X_2$; in either case, $z$ commutes with $X_3$. If, in addition, $z$ commutes with $X_2$, define $X_4=z$. Otherwise, $X_2z-zX_2=\gamma \Omega$, for some $\gamma \in \k^\times$; define $X_4=\alpha_{23}z - \gamma X_3$. As chosen, $X_4$ commutes with $X_2$ (and with $X_3)$, so we may always find an $X_4$ as described above.  It follows that $\Omega$ is normal in $A$.
\end{proof}

%%%%%%%%%%%%%%%%%%%%%%%%%%%%%%%%%%%%%%%%%%%%%%%%%%%%%%%%%%%%%%%%%%%%%
% END  PRELIMINARIES
% START  SECTION  2
%%%%%%%%%%%%%%%%%%%%%%%%%%%%%%%%%%%%%%%%%%%%%%%%%%%%%%%%%%%%%%%%%%%%%
\section{Regular Algebras whose Point Scheme is not the Union of Two Planes}\label{sec2}

In this section, we classify the AS-regular algebras $R \in \mathcal{R}(Q,\tau)$ whose point scheme $P$ strictly contains the rank-two quadric $Q$. We first address the case $\tau = \id$ in \Cref{PLAlg} and then proceed to a general automorphism $\tau$ by applying Zhang twists in \Cref{PLTwist}.

\begin{rmk}\label{HypsOnR}
Since $R$ is AS-regular with Hilbert series $H_R(t) = (1-t)^{-4}$, by \cite[Definition in \S 5.8, Corollary in \S 6.2]{Le}, we may assume $A$ is a quadratic noetherian algebra satisfying
\end{rmk}
\begin{enumerate}
\item $H_{R}(t) = (1-t)^{-4}$;
\item $R$ is Auslander-regular of global dimension four;
\item $R$ satisfies the Cohen-Macaulay property.
\end{enumerate}

\begin{lem}
Let $x_i$ and $x_j$ be two degree-one elements in $R$. If $x_i$ and $x_j$ are both normal in $R$, then $x_ix_j = \alpha_{ij} x_jx_i$, for some $\alpha_{ij} \in \k^\times$.
\end{lem}
\begin{proof}
Since $x_1$ and $x_2$ are two normal elements in the quadratic algebra $R$, it follows that $x_1x_2 \in R_1 x_1 \cap x_2 R_1$ and $x_2x_1 \in R_1 x_2 \cap x_1 R_1$; hence, $x_1x_2 \in \k^{\times} x_2 x_1$ and $x_2x_1 \in \k^{\times} x_1 x_2$ as $R$ is a domain.    
\end{proof}

\begin{lem}\cite[Lemma 2.2 \& Proposition 2.4]{VVr}\label{normalElt}
If the point scheme $P$ is not the quadric $Q$, then $W'$ is a five-dimensional subspace of $W$. Moreover, there exists a normal element in $R_1$.
\end{lem}

We note that by \cite[Proposition 2.4]{VVr}, $W_p=span\{ u^{\tau}\otimes x - x^{\tau} \otimes u, v^{\tau} \otimes x  - x^{\tau} \otimes v: x\in R_1\}$, where $u,v \in R_1$ are linearly independent with $\sigma(p), \tau(p), \tau\sigma\tau^{-1}(p) \in \V(u)$, $\sigma(p), \tau(p), \tau\sigma\tau^{-2}\sigma(p) \in \V(v)$, and $\sigma(p)$ denotes $\sigma_{P_{red}}(p)$.

\begin{prop}\cite[Lemma 2.5]{VVr}
\label{normalautom}
If $\omega \in R_1$ is normal and $\phi \in \Aut(R)$ is defined by $\omega x = x^{\phi}\omega$ for all $x \in R$, then $\phi \in \kk^{\times} \tau$.     
\end{prop}

%%%%%%%%%%%%%%%PROPOSITION 2.5%%%%%%%%%%%%%%%%%%%%%%%%%%%%%%%%%%%%%%%%%%%%%%%
\begin{prop}\cite[Theorem 2.6]{VVr}\label{centralElts}
If the point scheme $P$ of $R$ is not the quadric $Q$, then $\tau \in \Aut(R)$. In this case, $R$ is a Zhang twist of an AS-regular algebra $R'$ that maps onto the (commutative) homogeneous coordinate ring $S(\P^3)$. In particular, $R'$ contains two linearly independent central elements of degree one.
\end{prop}
\begin{proof}
If $P \ne Q$, then $P$ contains a normal element by Lemma \ref{normalElt}, so by Proposition \ref{normalautom} $\tau \in \Aut(R)$. As a result, one may twist $R$ by $\tau^{-1}$ to obtain an AS-regular algebra $R'$ of $\gldim$ 4 (see, Remark \ref{zhangtwist}). Thus, the automorphism $\tau'$ of $Q$ corresponding to $R'$ is $\tau \circ \tau^{-1}$, which is the identity (cf. \cite[\S 1.3]{SV99}). That is, $R'$ maps onto $S(\P^3)$. 

Now, let $v_1, v_2 \in R_1$ be linearly independent. By Lemma \ref{normalElt}, the span $W'$ of defining relations of $R'$ contains a five-dimensional subspace generated by $\{v_i \otimes x - x \otimes v_i: x \in R_1', i = 1, 2\}$, which proves that $v_1$ and $v_2$, are central.   
\end{proof}
%%%%%%%%%%%%%%%PROPOSITION 2.5%%%%%%%%%%%%%%%%%%%%%%%%%%%%%%%%%%%%%%%%%%%%%%%

\begin{rmk}
One may characterize any scheme as being reduced or nonreduced, (see e.g., \cite[Example II.9]{EH} and \cite[Definition \S II.3]{Hart}). We denote the reduced scheme of $P$ by $P_{\mbox{red}}$.
\end{rmk}

\begin{prop}\cite[Proposition 2.7]{VVr} \label{quadline}
If the point scheme $P$ of $R$ is not the quadric $Q$, then either
\begin{enumerate}[(a)]
\item	$P=\P^3$
\item	$P=Q \cup L$ where $L$ is a line in $\P^3$ that meets $Q$ in two points (counted with multiplicity), or 
\item	the reduced scheme $P_{red}$ associated to $P$ is the quadric $Q$ but the scheme $P$ contains a double line $L$ of (all) multiple points on $P$, where $L$ corresponds to a line on $Q$ (we denote this situation by $P=Q\uplus L$).		
\end{enumerate}	
	
In cases (b) and (c) there is a regular normalizing sequence $\{v_1,v_2\} \subset R_1$ such that $L=\V(v_1,v_2)$.
\end{prop}	
%\begin{proof}
%The proof is equivalent to that of Proposition 2.7 in \cite{VVr}.
%\end{proof}

% THEOREM 2.8
\begin{thm} \label{PLAlg}
Suppose that the point scheme $P$ of $R$ is not $Q$. If $\sigma$ denotes the automorphism on $P$ and $\sigma|_Q = \tau = \id$, then there is a choice of generators $x_1, x_2, x_3, x_4$ for $R$ such that $Q=\V(x_1x_2)$ with the following six quadratic defining relations.
\begin{enumerate}[(a)]
\item $R=S(\P^3)=\k[x_1,x_2,x_3,x_4]$ is the polynomial ring on four variables and $(P,\sigma)=(\P^3,\id)$.
\item	$R$ is the $\k$-algebra on generators $x_1,x_2,x_3,x_4$ with defining relations
$$\begin{matrix} x_3 x_1 = x_1 x_3, & x_4 x_2 = x_2 x_4, \\ x_3 x_2 = x_2 x_3, & x_4 x_1 = x_1 x_4, \\ x_3 x_4 = x_4 x_3, & x_2 x_1 = (\alpha + 1) x_1 x_2,  \end{matrix}$$
where $\alpha \in \k^\times \setminus\{-1\}$.  In this case, $P=Q\cup L$ where $L=\V(x_3,x_4)$ (so $L$ intersects $Q$ at two distinct points) and $\sigma|_L(x_1,x_2,0,0)=\left( (\alpha+1)x_1, x_2,0,0 \right)$.
\item	$R$ is the $\k$-algebra on generators $x_1,x_2,x_3,x_4$ with defining relations
$$\begin{matrix} x_3 x_1 = x_1 x_3, & x_2 x_1 = x_1 x_2, \\ x_3 x_2 = x_2 x_3, & x_2 x_4 - x_4 x_2 = x_1 x_2, \\ x_3 x_4 = x_4 x_3, & x_1 x_4 - x_4 x_1 = x_1 x_2.  \end{matrix}$$
In this case, $P=Q\cup L$ where $L=\V(x_1-x_2,x_3)$ (so $L$ intersects $Q$ at one point with multiplicity two) and $\sigma|_L(x_1,x_1,0,x_4)=\left(x_1,x_1,0, x_1+x_4\right)$.
\item	$R$ is the $\k$-algebra on generators $x_1,x_2,x_3,x_4$ with defining relations
$$\begin{matrix} x_1 x_2 = x_2 x_1, & x_2 x_3 = x_3 x_2, \\ x_1 x_3 = x_3 x_1, & x_2 x_4 = x_4 x_2, \\ x_1 x_4 = x_4 x_1, & x_3 x_4 - x_4 x_3 = x_1 x_2.  \end{matrix}$$
In this case, $P=Q\uplus L$ where $L=\V(x_1,x_2) \subset Q$ is the line of intersection of the planes of $Q$ and $\sigma \in \mathrm{Aut}(P)$ is uniquely determined by its action on the localized homogeneous coordinate rings of $P$, $S(P)[x_i^{-1}]$,  as follows: on $S(P)[x_1^{-1}]$ and $S(P)[x_2^{-1}]$ we have that $\sigma$ is the identity, but on $S(P)[x_3^{-1}]$ we have $$\sigma(x_1,x_2,x_3,x_4)=\left(x_1,x_2,x_3,x_4-x_3^{-1}x_1x_2\right),$$ and on $S(P)[x_4^{-1}]$ we have $$\sigma(x_1,x_2,x_3,x_4)=\left( x_1,x_2,x_3+x_4^{-1}x_1x_2,x_4 \right).$$
\item	$R$ is the $\k$-algebra on generators $x_1,x_2,x_3,x_4$ with defining relations
$$\begin{matrix} x_1 x_2 = x_2 x_1, & x_2 x_3 = x_3 x_2, \\ x_1 x_3 = x_3 x_1, & x_3 x_4 = x_4 x_3, \\ x_1 x_4 = x_4 x_1, & x_2 x_4 - x_4 x_2 = x_2 x_1.  \end{matrix}$$
In this case, $P=Q\uplus L$ where $L=\V(x_1,x_3) \subset \V(x_1) \subset Q$ and $\sigma \in \mathrm{Aut}(P)$ is uniquely determined by its action on the localized homogeneous coordinate  rings of $P$, $S(P)[x_i^{-1}]$, as follows: on $S(P)[x_1^{-1}]$ and $S(P)[x_3^{-1}]$ we have that $\sigma$ is the identity but on $S(P)[x_2^{-1}]$ we have $$\sigma(x_1,x_2,x_3,x_4)=\left( x_1,x_2,x_3,x_1+x_4 \right),$$ and on $S(P)[x_4^{-1}]$ we have $$\sigma(x_1,x_2,x_3,x_4)=(x_1,x_2-x_4^{-1}x_2x_1,x_3,x_4).$$
\end{enumerate}
\end{thm}	
\begin{proof}
Our hypotheses imply that $R$ maps onto the (commutative) homogeneous coordinate ring, $S(Q)$.  As a result, if $x_1, x_2, x_3$, and $x_4$ are generators for $S(Q)$ such that $Q=\V(x_1x_2)$, then the defining relations of $R$ must be nonzero linear combinations of $x_1x_2$ and $x_ix_j-x_jx_i$ for all $i,j$, by \Cref{omega}.
\begin{enumerate}[(a)]
\item If $P=\P^3$, then $\sigma \in \mathrm{Aut}(\P^3)$ is linear. Thus, $\sigma = \tau =$ identity and $R=S(\P^3)$.
\item	If $L$ is a line in $\P^3$ that meets $Q$ in two distinct points, then we may choose coordinates $x_1, x_2, x_3, x_4$ such that $L=\V(x_3,x_4)$ and $Q=\V(x_1x_2)$.  By Proposition \ref{centralElts}, we may choose $x_3$ and $x_4$ to be central elements in $R$ which yields five defining relations for $R$.  We know that $x_1x_2-x_2x_1 \neq 0$ in $R$ as $P\neq \P^3$. Thus, $\alpha x_1 x_2 = x_2x_1-x_1x_2$ for some $\alpha \in \k^\times$.  This yields the sixth defining relation, $x_2 x_1 = (\alpha+1) x_1 x_2$. We note that $\alpha \ne -1$, as $R$ is a domain.  However, if $\alpha\in \k^\times\setminus\{-1\}$, then the algebra is an Ore extension $\k[x_1,x_3,x_4][x_2;\varphi]$ where $\varphi(x_1)=(\alpha-1)x_1$, $\varphi(x_3)=x_3$ and $\varphi(x_4)=x_4$ and so, the algebra is regular of global dimension four. Moreover, the automorphism $\sigma|_L$ can be computed by factoring out the ideal $x_3R+x_4R$. 
				
\item	If $L$ is a line in $\P^3$ that meets $Q$ in one point counted with multiplicity two, then we choose coordinates $x_1, x_2, x_3, x_4$ such that $L=\V(x_1-x_2,x_3)$ and $Q=\V(x_1x_2)$.  By Proposition \ref{centralElts}, we may choose $x_1-x_2$ and $x_3$ to be central elements in $R$.  This gives seven possible defining relations: $x_3 x_i = x_i x_3$, for $i \in \{1,2,4\}$, and $(x_1-x_2)x_j = x_j(x_1-x_2)$, for $j\in \{1,2,3,4\}$. After eliminating some redundancies in the relations, we have the following four relations: 
\[\begin{matrix}
x_1x_2 = x_2x_1, & x_2x_3 = x_3x_2\\
x_1x_3 = x_3x_1, & x_3x_4 = x_4x_3
\end{matrix}\]
By \Cref{omega}, $\alpha x_1x_2 = x_1x_4-x_4x_1$. For the last relation, observe that $x_1x_4-x_4x_1=x_2x_4-x_4x_2$. This yields the final defining relation of $R$, $\alpha x_1x_2 = x_2x_4-x_4x_2$.  Mapping $x_1 \mapsto \sqrt{\alpha^{-1}}x_1$, $x_2 \mapsto \sqrt{\alpha^{-1}}x_2$, $x_3 \mapsto x_3$, $x_4 \mapsto \sqrt{\alpha}x_4$ gives the defining relations in (c). This algebra is the Ore extension $\k[x_1,x_2,x_3][x_4;\id,\delta]$ where $\delta(x_1)=-x_1x_2$, $\delta(x_2)=-x_1x_2$, $\delta(x_3)=0$; it may be verified that $\delta$ is a derivation and, hence, that $R$ is regular. We compute $\sigma|_L$ by factoring out the ideal $(x_1-x_2)R+x_3R$. 
				
\item	If $L$ is the line of intersection of the planes of $Q$, then there exists a choice of coordinates $x_1, x_2, x_3, x_4$ such that $L=\V(x_1,x_2)$ and $Q=\V(x_1x_2)$, and by \Cref{centralElts}, we may choose $x_1$ and $x_2$ to be central elements in $R$.   Thus, five of the six defining relations are determined.  It can be shown that $\alpha x_1 x_2 = x_3 x_4 - x_4 x_3 \in \k^\times \Omega$ for some $\alpha \in \k^\times$, which gives the final defining relation. Moreover, $x_1 \mapsto \sqrt{\alpha^{-1}}x_1$, $x_2 \mapsto \sqrt{\alpha^{-1}}x_2$, $x_3 \mapsto x_3$, $x_4 \mapsto x_4$ gives the desired defining relations.  This algebra is the Ore extension $\k[x_1,x_2,x_3][x_4;\id,\delta]$ where $\delta(x_1)=0$, $\delta(x_2)=0$, $\delta(x_3)=-x_1x_2$ and so, is regular.
				
To compute the action of the automorphism, we follow the methods in \cite[Proposition 2.6]{SV99} and set $Y_i'=x_i\otimes 1$ and $Z_i'=1\otimes x_i$ for all $i=1,...,4$ in the defining relations of the algebra.  This yields the (commutative) homogeneous coordinate ring $S$ corresponding to the point scheme $P$, $S(P)$, which is isomorphic to the commutative algebra $\k[Y_1',Y_2',Y_3',Y_4',Z_1',Z_2',Z_3',Z_4']$ with defining relations $$\begin{matrix} Y_1' Z_2' = Y_2' Z_1', & Y_2' Z_3' = Y_3' Z_2', \\ Y_1' Z_3' = Y_3' Z_1', & Y_2' Z_4' = Y_4' Z_2', \\ Y_1' Z_4' = Y_4' Z_1', & Y_3' Z_4' - Y_4' Z_3' = Y_1' Z_2'.  \end{matrix}$$
				
The subalgebra generated by the $Y_i'$ (respectively, the $Z_i'$) is the homogeneous coordinate ring of the left (respectively, right) projection of $P$ into $\P^3$ and is isomorphic to $S(P)$.  For each $S(P)[x_j^{-1}]$, set $Y_i=Y_i'/Y_j'$ and $Z_i=Z_i'/Z_j'$ for all $i \neq j$.  On each of these localizations, $\sigma$ is given by $\sigma(Y_i)=Z_i$.
				
For $S(P)[x_j^{-1}]$ where $j=1,2$, $\sigma$ is the identity.  For $S(P)[x_3^{-1}]$, we have $Y_2Z_1=Y_1Z_2$, $Y_4Z_1=Y_1Z_4$, $Y_4Z_2=Y_2Z_4$, $Z_1=Y_1$, $Z_2=Y_2$, $Z_4=Y_4-Y_1Z_2 = Y_4-Y_1Y_2$.  For $S(P)[x_4^{-1}]$, we have $Y_2Z_1=Y_1Z_2$, $Y_3Z_1=Y_1Z_3$, $Y_3Z_2=Y_2Z_3$, $Z_1=Y_1$, $Z_2=Y_2$, $Z_3=Y_3+Y_1Z_2=Y_3+Y_1Y_2$. The result follows.
				
\item	If $L$ is a line embedded in one of the planes of $Q$, then there exists a choice of coordinates $x_1, x_2, x_3, x_4$ such that $L=\V(x_1,x_3)$ and $Q=\V(x_1x_2)$.  By Proposition \ref{centralElts}, we may choose $x_1$ and $x_3$ to be central elements in $R$.   Thus, five of the six defining relations are determined and, it can be shown that $\alpha x_1 x_2 = x_2 x_4 - x_4 x_2$ for some $\alpha \in \k^\times$, which gives our final defining relation. Mapping $x_1 \mapsto \alpha^{-1}x_1$, $x_2 \mapsto x_2$, $x_3 \mapsto x_3$, $x_4 \mapsto x_4$ gives the desired defining relations.  This algebra is an Ore extension $\k[x_1,x_2,x_3][x_4;\id,\delta]$ where $\delta(x_1)=0$, $\delta(x_2)=-x_2x_1$, $\delta(x_3)=0$; it may be verified that $\delta$ is a derivation and, hence, that $R$ is regular. Moreover, retaining the same notation as above, the algebra $S(P)$ is isomorphic to the commutative algebra $\k[Y_1',Y_2',Y_3',Y_4',Z_1',Z_2',Z_3',Z_4']$ with defining relations $$\begin{matrix} Y_1' Z_2' = Y_2' Z_1', & Y_2' Z_3' = Y_3' Z_2', \\ Y_1' Z_3' = Y_3' Z_1', & Y_3' Z_4' = Y_4' Z_3', \\ Y_1' Z_4' = Y_4' Z_1', & Y_2' Z_4' - Y_4' Z_2' = Y_2' Z_1'.\end{matrix}$$				
				
For $S(P)[x_j^{-1}]$ where $j=1,3$, $\sigma$ is the identity.  For $S(P)[x_2^{-1}]$, we have $Y_3Z_1=Y_1Z_3$, $Y_4Z_1=Y_1Z_4$, $Y_4Z_3=Y_3Z_4$, $Z_1=Y_1$, $Z_3=Y_3$, $Z_4=Y_4+Z_1=Y_1+Y_4$.  For $S(P)[x_4^{-1}]$, we have $Y_2Z_1=Y_1Z_2$, $Y_3Z_1=Y_1Z_3$, $Y_3Z_2=Y_2Z_3$, $Z_1=Y_1$, $Z_3=Y_3$, $Z_2=Y_2-Y_2Z_1=Y_2-Y_2Y_1$.  The result follows.
\end{enumerate}	
\end{proof}	

\begin{cor}\cite[Corollary 2.13]{VVr}\label{tausigma}
The map $\tau \in \Aut(P)$ and it commutes with $\sigma$ on $P$.
\end{cor}
	
%THEOREM 2.10%%%%%%%%%%%%%%%%%%%%%%%%%%%%%%%%
\begin{thm}\label{PLTwist}
The regular algebra $R$ is isomorphic to one of the following algebras if and only if the point scheme $P$ of $R$ is not the quadric $Q$.
\begin{enumerate}[(a)]
    
    \item   $R$ is the $\k$-algebra on generators $x_1,x_2,x_3,x_4$ with defining relations $x_i^\tau x_j=x_j^\tau x_i$, for $1\leq i,j\leq 4$ for all $\tau \in \Aut(\P^3)$.  In this case, $R$ is a twist of the polynomial ring $S(\P^3)$ by $\tau$ so $(P,\sigma)=(\P^3,\tau)$.

    \item   $R$ is the $\k$-algebra on generators $x_1,x_2,x_3,x_4$ with defining relations $$\begin{matrix} x_3^\tau x_1 = x_1^\tau x_3, & x_4^\tau x_2 = x_2^\tau x_4, \\ x_3^\tau x_2 = x_2^\tau x_3, & x_4^\tau x_1 = x_1^\tau x_4, \\ x_3^\tau x_4 = x_4^\tau x_3, & x_2^\tau x_1 = (\alpha + 1) x_1^\tau x_2,  \end{matrix}$$ where $\alpha \in \k^\times\setminus \{-1\}$ and $\tau$ is given (with respect to a basis dual to $\{x_1, x_2, x_3, x_4\}$) by $$\tau \in \k^\times \begin{pmatrix} 1 & 0 & 0 & 0 \\ 0 & t_{22} & 0 & 0 \\ 0 & 0 & t_{33} & t_{34} \\ 0 & 0 & t_{43} & t_{44} \end{pmatrix},$$ where $t_{22} \in \k^\times$, $t_{ij}\in \k$ for $i,j\in \{3,4\}$, and $t_{33}t_{44}-t_{34}t_{43} \neq 0$.  In this case, $P=Q \cup L$ where $L=\V(x_3,x_4)$ (so $L$ intersects $Q$ at two distinct points), $\sigma|_Q=\tau$, and $\sigma|_L(x_1,x_2,0,0)=\tau\left( (\alpha+1)x_1,x_2,0,0\right)$.
                
    \item   $R$ is the $\k$-algebra on generators $x_1,x_2,x_3,x_4$ with defining relations $$\begin{matrix} x_3^\tau x_1 = x_1^\tau x_3, & x_2^\tau x_1 = x_1^\tau x_2, \\ x_3^\tau x_2 = x_2^\tau x_3, & x_2^\tau x_4 - x_4^\tau x_2 = x_1^\tau x_2, \\ x_3^\tau x_4 = x_4^\tau x_3, & x_1^\tau x_4 - x_4^\tau x_1 = x_1^\tau x_2,  \end{matrix}$$ where $\tau$ is given (with respect to a basis dual to the $x_i$) by $$\tau \in \k^\times \begin{pmatrix} 1 & 0 & 0 & 0 \\ 0 & 1 & 0 & 0 \\ t_{31} & -t_{31} & t_{33} & 0 \\ t_{41} & t_{42} & t_{43} & 1 \end{pmatrix} \text{ or } \tau \in \k^\times \begin{pmatrix} 0 & 1 & 0 & 0 \\ 1 & 0 & 0 & 0 \\ t_{31} & -t_{31} & t_{33} & 0 \\ t_{41} & t_{42} & t_{43} & 1 \end{pmatrix},$$ where $t_{33} \in \k^\times$, $t_{ij}\in \k$.  In this case, $P=Q \cup L$ where $L=\V(x_1-x_2,x_3)$ (so $L$ intersects $Q$ at one point with multiplicity two), $\sigma|_Q=\tau$, and $\sigma|_L(x_1,x_1,0,x_4)=\tau(x_1,x_1,0,x_1+x_4)$.
    
    \item   $R$ is the $\k$-algebra on generators $x_1,x_2,x_3,x_4$ with defining relations $$\begin{matrix} x_1^\tau x_2 = x_2^\tau x_1, & x_2^\tau x_3 = x_3^\tau x_2, \\ x_1^\tau x_3 = x_3^\tau x_1, & x_2^\tau x_4 = x_4^\tau x_2, \\ x_1^\tau x_4 = x_4^\tau x_1, & x_3^\tau x_4 - x_4^\tau x_3 = x_1^\tau x_2, \end{matrix}$$ where $\tau$ is given (with respect to a basis dual to the $x_i)$ by $$\tau \in \k^\times \begin{pmatrix} t_{11} & 0 & 0 & 0 \\ 0 & t_{22} & 0 & 0 \\ t_{31} & t_{32} & t_{33} & t_{34} \\ t_{41} & t_{42} & t_{43} & t_{44} \end{pmatrix} \text{ or } \tau \in \k^\times \begin{pmatrix} 0 & t_{12} & 0 & 0 \\ t_{21} & 0 & 0 & 0 \\ t_{31} & t_{32} & t_{33} & t_{34} \\ t_{41} & t_{42} & t_{43} & t_{44} \end{pmatrix} $$ where $t_{ij} \in \k$ and $t_{12}t_{21}+t_{11}t_{22}=t_{33} t_{44}-t_{34} t_{43} \neq 0$.  In this case, $P=Q\uplus L$ where $L=\V(x_1,x_2) \subset Q$ is the line of intersection of the planes of $Q$, $\sigma|_Q=\tau$, and $\sigma \in \mathrm{Aut}(P)$ is uniquely determined by its action on the local rings $S(P)[x_i^{-1}]$, where $S(P)$ is the homogeneous coordinate ring of $P$, as follows: on $S(P)[x_1^{-1}]$ and $S(P)[x_2^{-1}]$ we have that $\sigma = \tau^{-1}$, but on $S(P)[x_3^{-1}]$ we have $$\sigma(x_1,x_2,x_3,x_4)=\tau^{-1}\left(x_1,x_2,x_3,x_4-x_3^{-1}x_1x_2\right),$$ and on $S(P)[x_4^{-1}]$ we have $$\sigma(x_1,x_2,x_3,x_4)=\tau^{-1}\left( x_1,x_2,x_3+x_4^{-1}x_1x_2,x_4 \right).$$
                
    \item   $R$ is the $\k$-algebra on generators $x_1,x_2,x_3,x_4$ with defining relations
	$$\begin{matrix} x_1^\tau x_2 = x_2^\tau x_1, & x_2^\tau x_3 = x_3^\tau x_2, \\ x_1^\tau x_3 = x_3^\tau x_1, & x_3^\tau x_4 = x_4^\tau x_3, \\ x_1^\tau x_4 = x_4^\tau x_1, & x_2^\tau x_4 - x_4^\tau x_2 = x_2^\tau x_1.  \end{matrix}$$ where $\tau$ is given (with respect to a basis dual to the $x_i)$ by $$\tau \in \k^\times \begin{pmatrix} 1 & 0 & 0 & 0 \\ 0 & t_{22} & 0 & 0 \\ t_{31} & 0 & t_{33} & 0 \\ t_{41} & t_{42} & t_{43} & 1 \end{pmatrix}$$ where $t_{ij}\in \k$ for $i\neq j$ and $t_{22},t_{33} \in \k^\times$. 
    In this case, $P=Q\uplus L$ where $L=\V(x_1,x_3) \subset \V(x_1) \subset Q$, $\sigma|_Q=\tau$, and $\sigma \in \mathrm{Aut}(P)$ is uniquely determined by its action on the local rings $S(P)[x_i^{-1}]$, where $S(P)$ is the homogeneous coordinate ring of $P$, as follows: on $S(P)[x_1^{-1}]$ and $S(P)[x_3^{-1}]$ we have that $\sigma=\tau^{-1}$, but on $S(P)[x_2^{-1}]$ we have $$\sigma(x_1,x_2,x_3,x_4)=\tau^{-1}\left( x_1,x_2,x_3,x_1+x_4 \right),$$ and on $S(P)[x_4^{-1}]$ we have $$\sigma(x_1,x_2,x_3,x_4)=\tau^{-1}(x_1,x_2-x_4^{-1}x_2x_1,x_3,x_4).$$ 
\end{enumerate}
\end{thm}	
	
\begin{proof}
Using methods similar to \cite[Theorem 2.8]{SV99}, we classify the graded degree zero automorphisms of the algebras listed in Proposition \ref{PLAlg}.  Applying Corollary \ref{tausigma} to (a), (b) and (c) results in a matrix for $\tau$ that is determined by the hypotheses that $\tau\in \Aut(Q) \cap \Aut(L)$ and $\tau \mid_L \circ \ \sigma \mid_L = \sigma \mid_L \circ \ \tau \mid_L$ (since $\sigma|_Q=\id$).  The result follows by a computation, using Wolfram's Mathematica.
    
For (d) and (e), we apply $\tau$ to the defining relations of the respective algebras in Proposition \ref{PLAlg} rather than working with the commutativity of $\tau$ and $\sigma$.
    
    In the case of (d), any element of $\Aut(Q) \cap \Aut(L)$ may be given (with respect to a basis dual to the $x_i$) by $$\tau \in \k^\times \begin{pmatrix} t_{11} & t_{12} & 0 & 0 \\ t_{21} & t_{22} & 0 & 0 \\ t_{31} & t_{32} & t_{33} & t_{34} \\ t_{41} & t_{42} & t_{43} & t_{44} \end{pmatrix},$$ where $t_{ij} \in \k$, $t_{11}t_{21}=t_{12}t_{22}=0$ and $(t_{12}t_{21}-t_{11}t_{22})(t_{34}t_{43}-t_{33}t_{44}) \neq 0$.  This matrix maps the span of the defining relations of the algebra back to itself if and only if $t_{12}t_{21}+t_{11}t_{22}+t_{34}t_{43}-t_{33}t_{44}=0$, so the result follows.

    In the case of (e), any element of $\Aut(Q) \cap \Aut(L)$ may be given (with respect to a basis dual to the $x_i$) by $$\tau \in \k^\times \begin{pmatrix} t_{11} & 0 & 0 & 0 \\ 0 & t_{22} & 0 & 0 \\ t_{31} & 0 & t_{33} & 0 \\ t_{41} & t_{42} & t_{43} & t_{44} \end{pmatrix},$$ where $t_{ij} \in \k$ for $i \neq j$ and $t_{ii} \in \k^\times$.  This matrix maps the span of the defining relations of the algebra back to itself if and only if $t_{11}t_{22} = t_{22} t_{44}$, so the result follows.
\end{proof}
%THEOREM 2.10%%%%%%%%%%%%%%%%%%%%%%%%%%%%%%%%	
	
%%%%%%%%%%%%%%%%%%%%%%%%%%%%%%%%%%%%%%%%%%%%%%%%%%%%%%%%%%%%%%%%%%%%%
% END  SEC 1
% START  SEC 2
%%%%%%%%%%%%%%%%%%%%%%%%%%%%%%%%%%%%%%%%%%%%%%%%%%%%%%%%%%%%%%%%%%%%%

\section{Regular Algebras whose Point Scheme is the Union of Two Planes}\label{sec3}

In this section, we classify the AS-regular algebras $R \in \mathcal{R}(Q,\tau)$ whose point scheme is exactly the rank-two quadric $Q$. We show in \Cref{prop1} that such algebras always contain a normal element of degree one, and as a consequence, they are central extensions of AS-regular algebras of global dimension three.

To motivate this section, we begin with an example from \cite[Lemma 5.1.1]{HTDiss} where the second author showed that there exist examples of AS-regular algebras of global dimension four such that their point schemes is exactly a rank-two quadric.

\begin{ex}\cite[Lemma 5.1.1]{HTDiss}\label{hung}
Consider the $\Bbbk$-algebra $R$ defined by $T(V)/W$ where $T(V)$ is the tensor algebra, $V$ is the vector space given by $V = \Bbbk x_1 \oplus \Bbbk x_2 \oplus \Bbbk x_3 \oplus \Bbbk x_4$ and $W$ is the ideal generated by the six relations given below.
\begin{align*}
x_{1} x_{2} - x_{2}x_{1}, \hspace{4mm}
& x_{2} x_{3} + x_{3}x_{2},\\
x_{1} x_{3} - x_{3}x_{1}, \hspace{4mm}
& x_{2} x_{4} + x_{4}x_{2}, \\
x_{1} x_{4} - x_{4}x_{1},\hspace{4mm}
&x_{3} x_{4} - x_{4}x_{3} +  x_{1} x_{2}.  \\
\end{align*}
\end{ex}

The point scheme of the algebra $R$ in Example \ref{hung} is exactly a rank-two quadric. In fact, \Cref{prop1} below shows that when the point scheme of an algebra, satisfying the hypotheses of \Cref{HypsOnR}, is exactly a rank-two quadric and that the planes of $Q$ are invariant under the map $\tau$ (see e.g., Remark \ref{omega}(a)), then a normal degree-one element always exists in $R$. In Theorem \ref{prop1} below, it is worth noting that we do not classify all such possible maps $\tau \in \Aut(Q)$.

\begin{thm}\label{prop1}
Suppose that the automorphism $\tau \in \Aut(Q)$ preserves the two planes of $Q$. If the point scheme $P$ of $R$ is $Q$, then there exists generators $x_1, \dots, x_4$ for $R$ such that $P = Q = \V(x_1x_2)$ and $R$ is isomorphic to one of the following algebras.

\begin{enumerate}[(a)]

\item   $R$ is the $\k$-algebra on generators $x_1,x_2,x_3,x_4$ with defining relations

\begin{align*}
\alpha x_1^{\tau} x_2 = x_2^{\tau} x_1, \hspace{4mm}
& x_2^{\tau} x_3 = x_3^{\tau} x_2, \\ 
x_1^{\tau} x_3 = x_3^{\tau} x_1, \hspace{4mm}
&x_2^{\tau} x_4 = x_4^{\tau} x_2, \\ 
x_1^{\tau} x_4 = x_4^{\tau} x_1, \hspace{4mm}
& x_3^{\tau} x_4 - x_4^{\tau} x_3 = \beta x_1^{\tau} x_2,  \end{align*}

where $\alpha,\beta \in \k^\times$ and $\alpha\neq 1$.

\item   $R$ is the $\k$-algebra on generators $x_1,x_2,x_3,x_4$ with defining relations

\begin{align*}
x_1^\tau x_2 = x_2^\tau x_1, \hspace{4mm} & x_1^\tau x_3 = x_3^\tau x_1,\\  x_2^\tau x_3 - x_3^\tau x_2 = \alpha x_1^\tau x_2, \hspace{4mm} & x_1^\tau x_4 - x_4^\tau x_1 = \beta x_1^\tau x_2, \\x_2^\tau x_4 = x_4^\tau x_2, \hspace{4mm} & x_3^\tau x_4 = x_4^\tau x_3,
\end{align*}

where $\alpha, \beta \in \k^\times$.

\end{enumerate}
\end{thm}
\begin{proof}

In this proof, we retain notation used in the proof of \Cref{OmegaNormal}.

\begin{enumerate}[(a)]

\item   By Remark \ref{omega} (a), we may choose $x_1, x_2, x_3$ such that $x_1^{\tau} x_3 = x_3^{\tau} x_1$ and $x_2^{\tau} x_3 = x_3^{\tau} x_2$ and moreover, we take $\Omega =\delta x_1^{\tau} x_2 $. Thus, we write $\alpha x_1^{\tau} x_2 = x_2^{\tau} x_1$ where $\alpha = 1-\alpha_{12}\delta$, $x_i^{\tau} x_4 - x_4^{\tau} x_i = \alpha_i x_1^{\tau} x_2$, and $i\in \{1,2,3\}$ with $\alpha_i \in \k$.

If $\alpha = 1$, this algebra is isomorphic to the $\k$-algebra $A$ on generators $y_1,y_2,y_3,y_4$ with defining relations $y_1y_3=y_3y_1$, $y_2y_3=y_3y_2$, $y_1y_2=y_2y_1$ and $y_i y_4 - y_4 y_i = \alpha_i y_1 y_2$, where $i\in \{1,2,3\}$ and $\alpha_i \in \k$. The algebra $A$ is an Ore-extension of a polynomial ring on three variables given by the automorphism $\sigma(y_i)=y_i$, for all $i$, and the derivation $\gamma(y_1)=\gamma(y_2)=0$, $\gamma(y_3)=-\beta y_1y_2$ .  If, furthermore, $\alpha_i = 0$, for all $i$, then $A$ is the polynomial ring on four generators and $P=\P^3$.  If $\alpha_i \neq 0$ for some $i$, then $\alpha_i y_j - \alpha_j y_i$ are central, for $i,j \in \{1,2,3\}$ and so, by \Cref{quadline}, $Q \subset P$, which contradicts our hypothesis. Hence, we may assume $\alpha \neq 1$.  We rechoose $x_4$ such that $x_1^{\tau}x_4=x_4^{\tau}x_1$ and $x_2^{\tau}x_4 = x_4^{\tau} x_2$.  If, additionally, $x_3^{\tau} x_4 = x_4^{\tau} x_3$, then $R$ is a twist of the polynomial ring and its point scheme is not $Q$. Therefore, $x_3^{\tau}x_4 - x_4^{\tau} x_3 = \beta x_1^{\tau} x_2$, for $\beta \in \k^\times$.

\item If we choose $x_1,x_2,x_3,x_4$ such that $x_1^\tau x_2 = x_2^\tau x_1$, $x_1^\tau x_3 = x_3^\tau x_1$, $x_2^\tau x_3 - x_3^\tau x_2 = \alpha_{23} x_1^\tau x_2$, $x_2^\tau x_4 = x_4^\tau x_2$, and $x_3^\tau x_4 = x_4^\tau x_3$, where $\alpha_{23} \in \k^\times$, then the sixth relation is of the form $x_1 x_4 - x_4 x_1 = \alpha_{14} x_1x_2$, where $\alpha_{14} \in \k$. We write these relations as $Mx = 0$ where $M$ is a $6 \times 4$ matrix and $x^T = (x_1, x_2, x_3, x_3)$ and obtain the reduced scheme $P_{red}=Q$ where the ideal of $P_{red}$ is given by fifteen polynomials. Now, to determine the existence of points with multiplicity greater than one, by \cite[Definition \S I.5 \& Theorems 5.1 \& 5.3\footnote{While \cite[Theorem 5.1]{Hart} pertains to affine varieties, we know, using \cite[Corollary 2.3]{Hart}, that a projective variety can be covered by open sets which are then homeomorphic to affine varieties.}]{Hart}, we compute the $1\times 1$ minors of the Jacobian matrix of the fifteen polynomials of the ideal of $P_{red}$. This yields the possible points of higher multiplicity.

If $\alpha_{14} \neq 0$, then the $1\times 1$ minors all vanish on $\V(x_1,x_2) \subset Q$ which implies that points of the form $(0, 0, a, b) \in \P^3$ might be of higher multiplicity. Next, we apply Bertini's Theorem \cite[Theorem 8.18]{Hart}, and intersect $Q$ with a complementary-dimensional scheme, that is, a 1-dimensional scheme, which is given by the line $L_n=\V(x_1-x_2,n x_3 - x_4)$, where $n\in \P^1$. The point of intersection of $L_n$ and $\V(x_1,x_2)$ is $(0,0,1,n)$, when $n \in \k$, or $(0,0,0,1)$, when $n$ is a point at infinity. We determine the vector space dimension of the localization of the local ring $\mathcal{O}_{Q \cap L_n}$ at each of these points to compute their multiplicities \cite{Hart}. We obtain a ring isomorphic to $\k[x]/\langle x^2 \rangle$ and so, the point has multiplicity two. This implies that the points on $\V(x_1,x_2)$ are multiple points only as a consequence of being in the intersection of two irreducible components of $P$, $\V(x_1)$ and $\V(x_2)$, explicitly. Taking $\alpha=\alpha_{23}$ and $\beta=\alpha_{14}$ gives the defining relations as described.  This algebra is an extension of the polynomial ring $\k[x_2,x_3,x_4]$ by the normal element $x_1$ and is thus, regular.

Otherwise, $\alpha_{14}=0$, and we show that, in this case, we have a contradiction. Observe that the $1 \times 1$ minors of the corresponding Jacobian matrix vanish on $\V(x_1,x_2)$ and on $\V(x_1,x_4)$.  Using $L_n$ to examine the multiplicity of the points $\V(x_1,x_2)$ as before, we find that all points on $\V(x_1,x_2)$ have multiplicity two (as a consequence of the intersection of the irreducible components of $P$) except the point $(0,0,1,0)$, which has multiplicity three. To further analyze the point $(0, 0, 1, 0)$, consider the points on $\V(x_1,x_4)\setminus\V(x_1,x_2)$ that are of the form $(0,1,n,0)$, where $n\in \k$.  In a similar manner as above, we examine the multiplicity of these points using the line $K_n=\V(x_3- n x_2,x_1-x_4)$.  In doing so, we find that the coordinate ring of $P \cap K_n$ is isomorphic to $\k[x]/\langle x^2 \rangle$ and so, the points have multiplicity two. This implies that the point $(0, 0, 1, 0)$ which lies on the line $\V(x_1,x_4)$ is a point on an embedded line in $Q$. Thus, the point scheme of $R$ is $Q \cup \V(x_1,x_4)$, which contradicts our hypothesis.
\end{enumerate}
\end{proof}
%THEOREM 3.2%%%%%%%%%%%%%%%%%%%%%%%%%%%%%%%%

\begin{rmk}\label{centralext}
If $P = Q$ and there is a normal element $z \in R_1$, then $z$ is also regular as $R$ is a domain. Moreover, by \cite[Theorem 8.16(ii)]{ATV2}, the normal element $z^{\tau} \in R_1$ for an automorphism $\tau$, such that the planes of $Q$ are invariant, is central in $R^{\tau}$. As a result, in this case, $R$ can also be viewed as a central extension of an AS-regular algebra of global dimension three. These central extensions are well-understood algebras (see e.g., \cite{lBSV}), and by Theorem \ref{prop1}, the algebra $R$ will always be such an algebra.
\end{rmk}

\begin{ex}
Let $\tau \in \Aut(R)$ be an automorphism under which the planes of $Q$ are invariant. With respect to a basis dual to the basis $\{x_1, x_2, x_3, x_4\}$, one may take the map $\tau$ in Theorem \ref{prop1}(a) to be 
\begin{align}
\label{mat1}
\k^\times \begin{pmatrix} t_{11} & 0 & 0 & 0 \\ 0 & t_{22} & 0 & 0 \\ 0 & 0 & t_{33} & t_{34} \\ 0 & 0 & t_{43} & t_{44} \end{pmatrix},
\end{align} 
where $t_{ij} \in \k$ and $t_{11}t_{22}=t_{33}t_{44}-t_{34}t_{43} \neq 0$.  If $\alpha = -1$, then $\tau$ may also be given by 
\begin{align}\label{mat2}
\tau \in \k^\times \begin{pmatrix} 0 & t_{12} & 0 & 0 \\ t_{21} & 0 & 0 & 0 \\ 0 & 0 & t_{33} & t_{34} \\ 0 & 0 & t_{43} & t_{44} \end{pmatrix},
\end{align}
where $t_{ij} \in \k$ and $t_{12}t_{21}=t_{34}t_{43}-t_{33}t_{44} \neq 0$.

If $\tau \in \Aut(Q)$, then $\tau$ is given by $$\begin{pmatrix} t_{11} & 0 & 0 & 0 \\ 0 & t_{22} & 0 & 0 \\ t_{31} & t_{32} & t_{33} & t_{34} \\ t_{41} & t_{42} & t_{43} & t_{44} \end{pmatrix} \text{ or }\begin{pmatrix} 0 & t_{12} & 0 & 0 \\ t_{21} & 0 & 0 & 0 \\ t_{31} & t_{32} & t_{33} & t_{34} \\ t_{41} & t_{42} & t_{43} & t_{44}\end{pmatrix}.$$  
The matrices of form described in \ref{mat1} maps the span of the defining relations of $R$ back to itself if and only if $t_{31}=t_{32}=t_{41}=t_{42}=0$ and $t_{11}t_{22}+t_{34}t_{43}-t_{33}t_{44}=0$. In addition, matrices of form given in \ref{mat2} maps the span of the defining relations of $R$ back to itself if and only if $(\alpha^2-1)t_{12}t_{21}=0$, $t_{31}=t_{32}=t_{41}=t_{42}=0$ and $t_{12}t_{21}-t_{34}t_{43}+t_{33}t_{44}=0$.  Since $\alpha \neq 1$ and $t_{12}t_{21} \neq 0$, it follows that $\alpha=-1$.
\end{ex}

\begin{rmk}
The algebra $R$ in Example \ref{hung} is isomorphic to an algebra given by Theorem \ref{prop1}(a). Taking $\alpha = -1 = \beta$ and defining $\tau$ by $x_i^{\tau} = x_i$ for all $i \neq 2$ and $x_2^{\tau} = -x_2$ (with respect to the basis $\{x_1, x_2, x_3, x_4\}$) yields exactly the algebra $R$ of Example \ref{hung}.
\end{rmk}

\medskip

\noindent Richard Chandler \url{https://orcid.org/0000-0002-3963-3230}

\noindent Hung V.\ Tran \url{https://orcid.org/0000-0003-4386-3239}

\noindent Padmini Veerapen \url{https://orcid.org/0000-0002-9273-8904}

\noindent Xingting Wang \url{https://orcid.org/0000-0002-6485-0472}

\bibliographystyle{plain}
\bibliography{Revision2JAAFinalArXiv}

\end{document}